\documentclass[letterpaper]{article}

\usepackage[letterpaper,left=1.3in,right=1.3in,top=1.5in,bottom=1.5in]{geometry}
\usepackage{bbm}
\usepackage{mathtools}

\newcommand{\myauthor}{Benjamin Antieau, Akhil Mathew, and Thomas Nikolaus}
\newcommand{\mytitle}{On the Blumberg--Mandell K\"unneth theorem for TP}

\title{On the Blumberg--Mandell K\"unneth theorem for $\TP$}

\author{Benjamin Antieau\footnote{Benjamin Antieau was supported
by NSF Grant DMS-1552766.}, Akhil Mathew,\footnote{Akhil
Mathew was supported by a Clay Research Fellowship.}\, and Thomas Nikolaus}
\date{\today}
 
\usepackage[pdfstartview=FitH,
            pdfauthor={\myauthor},
            pdftitle={\mytitle},
            colorlinks,
            linkcolor=reference,
            citecolor=citation,
            urlcolor=e-mail,
            backref]{hyperref}

\usepackage{amsmath}
\usepackage{amscd}
\usepackage{amsbsy}
\usepackage{amssymb}
\usepackage{verbatim}
\usepackage{eufrak}
\usepackage{eucal}
\usepackage{microtype}
\usepackage{hyperref}
\usepackage{mathrsfs}
\usepackage{amsthm}
\usepackage[all,cmtip]{xy}
\usepackage{tikz}
\usetikzlibrary{matrix,arrows}
\usepackage{dsfont}

\usepackage{fancyhdr}
\usepackage{makeidx}
\makeindex

\usepackage{color}
\definecolor{todo}{rgb}{1,0,0}
\definecolor{conditional}{rgb}{0,1,0}
\definecolor{e-mail}{rgb}{0,.40,.80}
\definecolor{reference}{rgb}{.20,.60,.22}
\definecolor{mrnumber}{rgb}{.80,.40,0}
\definecolor{citation}{rgb}{0,.40,.80}

\let\oldmarginpar\marginpar
\renewcommand\marginpar[1]{\-\oldmarginpar[\raggedleft\footnotesize #1]%
{\raggedright\footnotesize #1}}

\newcommand{\Ascr}{\mathcal{A}}

\newcommand{\Cscr}{\mathcal{C}}
\newcommand{\Dscr}{\mathcal{D}}

\newcommand{\E}{\mathrm{E}}

\renewcommand{\H}{\mathrm{H}}

\renewcommand{\L}{\mathrm{L}}

\renewcommand{\mathds}{\mathbb}

\newcommand{\EE}{\mathds{E}}
\newcommand{\FF}{\mathds{F}}

\newcommand{\Sp}{\mathrm{Sp}}
\newcommand{\CycSp}{\mathrm{CycSp}}

\newcommand{\ev}{\mathrm{ev}}
\newcommand{\coev}{\mathrm{coev}}

\DeclareMathOperator{\id}{id}

\renewcommand{\geq}{\geqslant}
\renewcommand{\leq}{\leqslant}

\newcommand{\THH}{\mathrm{THH}}
\newcommand{\HP}{\mathrm{HP}}
\newcommand{\TP}{\mathrm{TP}}
\newcommand{\TC}{\mathrm{TC}}
\newcommand{\TF}{\mathrm{TF}}
\newcommand{\TR}{\mathrm{TR}}
\newcommand{\HH}{\mathrm{HH}}

\DeclareMathOperator{\Pic}{Pic}

\DeclareMathOperator{\End}{End}

\DeclareMathOperator{\Hom}{Hom}

\newcommand{\Mod}{\mathrm{Mod}}

\newcommand{\Perf}{\mathrm{Perf}}

\newcommand{\CAlg}{\mathrm{CAlg}}

\newcommand{\we}{\simeq}
\newcommand{\iso}{\cong}

\theoremstyle{plain}
\newtheorem{theorem}{Theorem}[section]
\newtheorem*{theorem*}{Theorem}
\newtheorem{lemma}[theorem]{Lemma}

\newtheorem{proposition}[theorem]{Proposition}

\newtheorem{corollary}[theorem]{Corollary}

\newtheoremstyle{named}{}{}{\itshape}{}{\bfseries}{.}{.5em}{#1 \thmnote{#3}}
\theoremstyle{named}

\theoremstyle{definition}
\newtheorem{definition}[theorem]{Definition}

\newtheorem{example}[theorem]{Example}

\newtheorem{construction}[theorem]{Construction}
\newtheorem{remark}[theorem]{Remark}

\begin{document}

\maketitle

\begin{abstract}
    \noindent
    We give a new proof of the recent K\"unneth theorem for periodic topological cyclic
    homology ($\TP$) of smooth and proper dg categories over perfect fields of
    characteristic $p>0$ due to Blumberg and
    Mandell. Our result is slightly stronger and implies a
    finiteness theorem for topological cyclic homology ($\TC$) of such categories.


    \paragraph{Mathematics Subject Classification 2010.}
    \href{http://www.ams.org/mathscinet/msc/msc2010.html?t=14Fxx&btn=Current}{14F30},
    \href{http://www.ams.org/mathscinet/msc/msc2010.html?t=16Exx&btn=Current}{16E40},
    \href{http://www.ams.org/mathscinet/msc/msc2010.html?t=19Dxx&btn=Current}{19D55}.
\end{abstract}


\section{Introduction}

Let $k$ be a perfect field and let $A$ be a commutative $k$-algebra. 
In \cite{hesselholt-tp}, Hesselholt studies the 
\emph{periodic topological cyclic homology} $\TP(A)$, defined as the Tate
construction $\THH(A)^{tS^1}$ of the $S^1$-action on the topological Hochschild
homology spectrum $\THH(A)$.\footnote{More generally, if $R$ is any $\EE_\infty$-ring spectrum and $A$ is an
$R$-algebra, we can consider $\THH(A/R)$, which is topological Hochschild
homology relative to $R$. If $R$ is discrete, we write $\HH(A/R)=\THH(A/R)$.}
In characteristic $0$,  the analogous construction, periodic cyclic homology
$\HP(\cdot/k)=\HH(\cdot/k)^{tS^1}$, is related to (2-periodic) de Rham cohomology. 
In characteristic $p>0$, which we will assume henceforth, $\TP(A)$ is of significant
arithmetic interest: by work of Bhatt, Morrow, and Scholze
\cite{BMS2}, the
construction $\TP(A)$ is closely related to the \emph{crystalline cohomology}
of $A$ over $k$.  For instance, one has $\pi_*(\TP(k))\simeq  W(k)[x^{\pm 1}]$ with $|x|
= 2$, a $2$-periodic form of the coefficient ring of crystalline cohomology.

The construction $A \mapsto\TP(A)$ globalizes to quasi-compact quasi-separated schemes, and in fact $\TP$ of
a scheme $X$ is determined
entirely by the dg category of perfect complexes on $X$.
Let $\Cscr$ be a $k$-linear dg category. In this case, one similarly defines
the spectrum $\TP(\Cscr) =\THH(\mathcal{C})^{tS^1}$. Thus $\TP$ defines a 
functor from $k$-linear dg categories to $\TP(k)$-module spectra. 
In a similar way that 
crystalline cohomology is a lift to characteristic zero of de Rham cohomology,
$\TP(\mathcal{C})$ is an integral 
lift of  the periodic cyclic homology $\HP(\mathcal{C}/k)$
(see Theorem~\ref{integrallift}, which is due to \cite{BMS2}).
 With respect to the tensor product on $k$-linear dg categories, the construction $\mathcal{C} \mapsto
\TP(\mathcal{C})$ is a lax symmetric monoidal functor. 

Many cohomology theories for schemes, such as the crystalline theory mentioned
above, satisfy a K\"unneth
formula. In~\cite{blumberg-mandell-tp}, Blumberg and Mandell prove the
following result.
\begin{theorem}[Blumberg--Mandell] 
\label{BMthm}
Let $k$ be a perfect field of characteristic $p > 0$. 
\begin{enumerate}
\item  
If $\Cscr$ is a smooth and proper $k$-linear dg category, then
$\TP(\Cscr)$ is compact as a $\TP(k)$-module spectrum.
\item
(K\"unneth formula) If $\Cscr$ and $\Dscr$ are smooth and proper $k$-linear dg
categories, then the natural map
$$\TP(\Cscr)\otimes_{\TP(k)}\TP(\Dscr)\rightarrow\TP(\Cscr\otimes_k\Dscr)$$ is
an equivalence.
\end{enumerate}
\end{theorem} 
  
This result is the key ingredient in Tabuada's proof \cite{tabuada} of the fact that the category of noncommutative
numerical motives is abelian semi-simple. This is a generalization of Jannsen's theorem \cite{jannsen}  to the noncommutative case.
With this application in mind the above theorem was suggested by Tabuada. \\

In this paper, we give a short proof of a stronger form of 
Theorem~\ref{BMthm}.  Let $(\mathcal{A}, \otimes, \mathbbm{1})$ be a symmetric
monoidal, stable $\infty$-category with biexact tensor product (in short:
\emph{a stably symmetric monoidal $\infty$-category}). 
We now use the following definition. 

\begin{definition}\label{perfect}
    An object $X \in \mathcal{A}$ is \emph{perfect} if it belongs to the thick subcategory generated by the unit.
\end{definition}

Note that for $\mathcal{A}\we\Mod_R$, the $\infty$-category of modules over
an $\EE_\infty$-ring
spectrum $R$, an object is perfect if and only if it is dualizable if and only
if it is compact. For general $\Ascr$, this is not typically the case.

Let $\Sp$ be the $\infty$-category of
spectra and let $F \colon  \mathcal{A} \to \Sp$ be a lax
symmetric monoidal, exact functor. 
Note that $F(\mathbbm{1}) $ is naturally an $\EE_{\infty}$-ring and $F(X)$ is an
$F(\mathbbm{1})$-module for any $X \in \mathcal{A}$. 
For any $X, Y \in \mathcal{A}$, we have a natural map
\begin{equation} \label{compmap}  F(X) \otimes_{F(\mathbbm{1})} F(Y) \to F(X \otimes Y).
\end{equation}
If $X$ is perfect, then, by a thick subcategory argument, 
\eqref{compmap} is an equivalence for every $Y \in \mathcal{A}$ and 
$F(X)$ is perfect as an $F(\mathbbm{1})$-module.

We will be interested in the following example: let $\Sp^{B S^1}$ denote the
$\infty$-category of spectra equipped with an $S^1$-action. Then $\THH(k)$
naturally defines a commutative algebra object in 
$\Sp^{B S^1}$ and we take $\mathcal{A} = \Mod_{\THH(k)}(\Sp^{B S^1})$. 
When $\mathcal{C}$ is any $k$-linear dg category, $\THH(\mathcal{C})$
defines an object of $\mathcal{A}$. We then consider the functor $F \colon  \mathcal{A} \to \Sp$ sending $X
\mapsto X^{tS^1}$. With this in mind, Theorem~\ref{BMthm} is implied by
the following result.
\begin{theorem} 
\label{perfectresult}
Let $k$ be a perfect field of characteristic $p > 0$. 
Any dualizable object in the symmetric monoidal $\infty$-category
$\Mod_{\THH(k)}(\Sp^{B S^1})$ is perfect. 
In particular, if $\mathcal{C}$ is a smooth and proper dg category over $k$,
then $\THH(\mathcal{C}) \in \Mod_{\THH(k)}(\Sp^{B S^1})$ is perfect. 
\end{theorem} 

We note also that our result implies that in Theorem~\ref{BMthm}, one
only needs one of $\mathcal{C}, \mathcal{D}$ to be smooth and proper, and one
can replace $\TP$ with $\THH^{tH}$ or with $\THH^{hH}$ for any closed subgroup
$H$ of $S^1$. In particular, taking $H=S^1$, we obtain a K\"unneth theorem for $\TC^{-} \stackrel{\mathrm{def}}{=}\THH^{hS^1}$.
We deduce below Theorem~\ref{perfectresult} directly
from the \emph{regularity} of $\pi_*\THH(k)$.
The argument also works for dg categories
over the localization $A$ of a ring of integers in a number field at a
prime over $p$ but with $\THH$ replaced with $\THH$ relative to $S^0[q]$ where
$S^0[q]\rightarrow HA$ sends $q$ to a chosen uniformizing parameter $\pi$
(see Corollary~\ref{cor:sq}).

We show, using the formula of Nikolaus--Scholze \cite{nikolaus-scholze}, that our strengthened version of the Blumberg--Mandell
theorem has an application to a finiteness theorem for topological cyclic
homology. When one works with smooth and proper schemes, stronger results are known by
work of Geisser and Hesselholt \cite{GH}, but our results seem to be new in
the generality of smooth and proper dg categories.

\begin{theorem} 
    If $\mathcal{C}$ is a smooth and proper dg category over a finite field $k$ of
    characteristic $p$, then $\TC(\mathcal{C})$ is perfect as an
    $H\mathbb{Z}_p$-module.\footnote{Note that $\TC(\mathcal{C})$ is an $H\mathbb{Z}_p$-module since $H\mathbb{Z}_p \simeq \tau_{\geq 0}\TC(k)$.}
\end{theorem} 

Finally, given the theorems above and the fact that for $\Cscr$ smooth and
proper $\THH(\Cscr)$ is dualizable
as a $\THH(k)$-module spectrum in the $\infty$-category $\CycSp$ of cyclotomic
spectra, a very natural question is to ask whether $\THH(\Cscr)$ is perfect as a $\THH(k)$-module in
$\CycSp$. We give an example to show that this is not the case.

\begin{theorem}
    If $X$ is a supersingular K3 surface over a perfect field $k$ of
    characteristic $p>0$, then $\THH(X)$ is not perfect in
    $\Mod_{\THH(k)}(\CycSp)$.
\end{theorem}


\subsection*{Acknowledgments}
We would like to thank Bhargav Bhatt, Lars Hesselholt, and Peter Scholze for helpful
discussions. 
We also thank Markus Land and the referee for extremely helpful and detailed
comments on earlier versions.
Finally, the second author thanks the University of Bonn for its hospitality during
a week long visit.

\section{Symmetric monoidal $\infty$-categories}\label{sec:prelim}

As indicated in the introduction, our approach is based on the notion of
\emph{perfectness} (Definition~\ref{perfect}) in a stably symmetric monoidal 
$\infty$-category.
Any perfect object is dualizable, but in general the converse need not hold. 
In this section, we give a criterion (Theorem~\ref{technicalperfect})
for when dualizable objects \emph{are} perfect in symmetric monoidal
$\infty$-categories of spectra with the action of
a compact connected Lie group.

\newcommand{\fun}{\mathrm{Fun}}

We will need some preliminaries on presentably symmetric monoidal
stable $\infty$-categories\footnote{By \emph{presentably symmetric monoidal} we mean that the underlying $\infty$-category is presentable and the tensor products preserves colimits in both variables seperately.}. We refer to \cite{MNN17} for more details. 
Let $(\mathcal{C}, \otimes, \mathbbm{1})$ be one such. 
In this case, one has an adjunction
\begin{equation}  \label{adj} \Mod_{
\mathrm{End}_{\mathcal{C}}(\mathbbm{1})} \rightleftarrows
\mathcal{C},  \end{equation}
where the left adjoint $\cdot
\otimes_{\mathrm{End}_{\mathcal{C}}(\mathbbm{1})}\mathbbm{1} \colon  
\Mod_{ \mathrm{End}_{\mathcal{C}}(\mathbbm{1})}
 \to
\mathcal{C}$ is symmetric monoidal and the right 
adjoint $\Hom_{\mathcal{C}}( \mathbbm{1}, \cdot) \colon  
\mathcal{C} \to \Mod_{ \mathrm{End}_{\mathcal{C}}(\mathbbm{1})}
$ is lax symmetric monoidal. 
Here, $\End_{\Cscr}(\mathbbm{1})$ is an $\EE_\infty$-ring spectrum and $\Mod_{ \mathrm{End}_{\mathcal{C}}(\mathbbm{1})}$ is the $\infty$-category
of $\mathrm{End}_{\mathcal{C}}(\mathbbm{1})$-modules in $\Sp$. 
One knows (cf. \cite{SchwedeShipley} and \cite[Section~7.1.2]{HA}) that the adjunction is an
equivalence precisely when $\mathbbm{1} \in \mathcal{C}$ is a compact generator. 
We need a definition which applies in some situations when $ \mathbbm{1}$ is a generator, but is no longer
required to be compact. 

\begin{definition}[{Cf. \cite[Def. 7.7]{MNN17}}] 
The presentably symmetric monoidal stable $\infty$-category
$\mathcal{C}$ is \emph{unipotent} if 
\eqref{adj} is a localization, i.e., if the counit map
$\Hom_{\mathcal{C}}(\mathbbm{1}, X)
\otimes_{\mathrm{End}_{\mathcal{C}}(\mathbbm{1}) } \mathbbm{1} \to X$ is an
equivalence for every $X \in \mathcal{C}.$
\end{definition} 

\begin{lemma}\label{lem:locuni}
    If $\mathcal{C}$ is a symmetric monoidal localization (see \cite[Definition
    2.2.1.6 and Example 2.2.1.7]{HA}) of the module category of an $\EE_\infty$-ring, then
    $\mathcal{C}$ is unipotent. It follows that if $\mathcal{C}$ is unipotent and if
    $A \in \CAlg(\mathcal{C})$, then $\Mod_A(\mathcal{C})$ is unipotent too. 
\end{lemma} 

\begin{proof}
    Suppose that $\Cscr$ is a symmetric monoidal localization of $\Mod_R$ where
    $R$ is an $\EE_\infty$-ring spectrum. Write $S$ for the $\EE_\infty$-ring
    spectrum $\End_\Cscr(\mathbbm{1})$. The adjunction~\eqref{adj}
    extends to the left to a sequence of adjunctions
    $$\Mod_R\rightleftarrows\Mod_{S}\rightleftarrows\Cscr,$$
    where $\Mod_R\rightarrow\Mod_{S}$ is extension of
    scalars along a map $R\rightarrow S$ of
    $\EE_\infty$-ring spectra. Since the composition of the right adjoints is
    fully faithful, and sends $\mathbbm{1}$ to $S$, viewed as an
    $R$-module, and because $\Mod_S\rightarrow\Mod_R$ preserves colimits, we see that the right adjoint $\Mod_S\rightarrow\Mod_R$ is
    fully faithful. This implies that $\Cscr\rightarrow\Mod_S$ is fully
    faithful. This proves the first statement. For the second, write $T$ for
    the spectrum of maps from $\mathbbm{1}$ to $A$ in $\Cscr$. This is an
    $\EE_\infty$-algebra over $S$ and $\Mod_A(\Cscr)$ is a localization of
    $\Mod_T(\Mod_S)\we\Mod_T$ if $\Cscr$ is unipotent.
\end{proof}

We will study unipotence in the following scenario. Let $X$ be a space. 
We can then form the presentable stable $\infty$-category 
$\Sp^X = \fun(X, \Sp)$ of spectra parametrized over $X$. 
We regard $\Sp^X$ as a symmetric 
monoidal $\infty$-category with the pointwise tensor product. 
Given an $\EE_\infty$-algebra $R$ in $\Sp^X$, we can
form the presentably symmetric monoidal stable $\infty$-category $\Mod_R(\Sp^X)$ of $R$-modules in
$\Sp^X$. When $R$ is given by a constant diagram, then there is a natural
equivalence $\Mod_R(\Sp^X)\we\fun(X, \Mod_R)$ and the endomorphisms of the unit are given by the function
spectrum $F(X_+;R)$.
In this case, one has the following unipotence criterion. 

\begin{theorem}[{Cf. \cite[Theorem 7.35]{MNN17}}] 
\label{unipthm}
Let $R$ be an $\EE_\infty$-ring such that $\pi_* (R)$ is concentrated in even
degrees.  Let $G$ be a compact, connected Lie group such that 
the cohomology $\H^*(BG; \pi_0 (R))$ is a polynomial algebra on
even-dimensional generators. 
In this case, the presentably symmetric monoidal $\infty$-category  $\fun(BG, \Mod_R)$ is unipotent. 
\end{theorem}

\begin{remark} 
    We refer also to \cite[Theorem 8.13]{MNN17} for an example (essentially due
    to Hodgkin, Snaith, and McLeod) where $\fun(BG, \Mod_R)$ is
    unipotent although the polynomiality condition above is far from satisfied. 
\end{remark} 


We will need a more precise version of this result, when $R$ is in
addition allowed to have a nontrivial $G$-action.

\begin{corollary}
\label{cor:nontrivunipotence}
Let $G$ be a compact, connected Lie group. 
Let $R \in \CAlg( \Sp^{BG})$ be an $\EE_\infty$-ring with $G$-action such that $\pi_* (R)$ is concentrated in even
degrees and   such that 
the cohomology
$\H^*(BG; \pi_0 (R))$ is a polynomial algebra over $\pi_0 (R)$ on even-dimensional generators. 
Then
$\Mod_R( \Sp^{BG})$ is unipotent. 
Furthermore, there are classes $x_1, \dots, x_t \in \pi_*(R^{hG})$ such that 
\begin{enumerate}
    \item $(x_1,\cdots,x_t)$ is a regular sequence in $\pi_*(R^{hG})$;
\item  $R^{hG}/(x_1, \dots, x_t) \simeq R$ and 
$ \pi_*( R^{hG})/(x_1, \dots, x_t) \simeq \pi_*(R)$;
\item
$\Mod_{R}( \Sp^{BG})$ is identified with the symmetric monoidal
$\infty$-category of $R$-complete objects in $\Mod_{R^{hG}}$. 
\end{enumerate}

\end{corollary} 

\begin{proof} 
We consider the basic adjunction
\begin{equation} \label{thisadjn} \Mod_{R^{hG}} \rightleftarrows \Mod_R(\Sp^{BG}),  \end{equation}
where the left adjoint sends $M \mapsto M \otimes_{R^{hG}} R$ for $M \in
\Mod_{R^{hG}}$ and the right
adjoint sends $V \mapsto V^{hG}$ for $V \in \Mod_R(\Sp^{BG})$. 
    Note that by Lemma~\ref{lem:locuni} if $\mathcal{C} $
    is unipotent and $A \in \CAlg(\mathcal{C})$, then $\Mod_A(\mathcal{C})$ is also
    unipotent.  Therefore, unipotence of $\Mod_R(\Sp^{BG})$ follows from 
    Theorem~\ref{unipthm} with $R$ replaced by $R^{hG}$ by using the
    equivalence $\Mod_R(\fun(BG,\Mod_{R^{hG}}))\we\Mod_R(\Sp^{BG})$ and
	thus \eqref{thisadjn} is a localization. 
Unwinding the definitions, it follows that the localization inverts those maps
of $R^{hG}$-modules $M \to M'$ such that $M \otimes_{R^{hG}} R \to M'
\otimes_{R^{hG}} R$ is an equivalence, which verifies statement 3. Statements 1
and 2 follow easily from the degenerate homotopy fixed point spectral sequence
for $\pi_*(R^{hG}$) and the assumptions. 
    \end{proof} 

Let $\mathcal{C}$ be a unipotent presentably symmetric monoidal stable $\infty$-category. 
Suppose $X \in \mathcal{C}$ is a \emph{dualizable} object. In this case, one
wants a criterion in order for $X$ to be perfect.  The next result follows from the statement that $X \simeq
\Hom_{\mathcal{C}}(\mathbbm{1}, X)
\otimes_{\mathrm{End}_{\mathcal{C}}(\mathbbm{1}) } \mathbbm{1}$ for any $X \in
\mathcal{C}$ when $\mathcal{C}$ is unipotent. 

\begin{proposition}\label{prop:unipotencecriterion}
    If $\mathcal{C}$ is a unipotent presentably symmetric monoidal stable
    $\infty$-category, then an object 
    $X \in \mathcal{C}$ is perfect if and only if $\Hom_{\mathcal{C}}(\mathbbm{1},
    X)$ is perfect (i.e., dualizable) as an $\mathrm{End}_{\mathcal{C}}(\mathbbm{1})$-module. 
\end{proposition} 

In general, dualizable objects need not be perfect without some type of regularity.
We illustrate this with two examples. 

\begin{example} 
Let $X \in \Mod_{H\mathbb{Z}_p}$ be $p$-complete. Then the following are equivalent: 
\begin{enumerate}
\item  
$X$ is perfect
in $\Mod_{H\mathbb{Z}_p}$.
\item
$X \otimes_{H \mathbb{Z}_p} H \mathbb{F}_p \simeq X/p \in \Mod_{ H
\mathbb{F}_p}$
is perfect. 
\item 
$X$ is dualizable in the $\infty$-category of $p$-complete $H
\mathbb{Z}_p$-modules. 
\item The direct sum $\bigoplus_{i \in \mathbb{Z}} \pi_i(X)$ is a finitely
generated $\mathbb{Z}_p$-module.
\end{enumerate}
Clearly 1 implies 3 which implies 2, and 4 implies 1.
Thus, it suffices to see that 2 implies 4, which follows because $\pi_* (X)$
is derived $p$-complete and $\pi_*(X)/p \pi_*(X) \subseteq \pi_* (X/p)$ is
finitely generated (cf. Lemma~\ref{lem:finitelygen} below for a more general
argument). 
\end{example} 

\begin{example} 
    Let $\mathcal{C} = L_{K(n)}
    \Sp$ be the $\infty$-category of $K(n)$-local spectra (with respect to some
    prime $p$), which is unipotent as a
    symmetric monoidal localization of $\Sp$ (see Lemma~\ref{lem:locuni}). 
    There are numerous invertible (hence, dualizable) objects in $\mathcal{C}$
    \cite{HMSPic} which are not perfect. 
	 
	 To see that these objects are not
    perfect, 
	we argue as follows. 
	Given $M \in \Pic( L_{K(n)} \Sp)$, we form the object $L_{K(n)} (E_n \wedge
	M)$ and take the homotopy groups, equipped with the action of the Morava
	stabilizer group, i.e., we can consider the Morava module. 
We specialize to the case $n = 1$: thus,  we take the $p$-adic $K$-theory of $M$ together with
the action of the Adams operations $\psi^l, l \in \mathbb{Z}_p^{\times}$. 
If $M$ is perfect, then a thick subcategory argument shows that the eigenvalues
of $\psi^l$ are integer powers of $l$. 
By contrast, this need not be the case for objects in $\Pic( L_{K(1)} \Sp)$, 
as follows from the construction in~\cite[Proposition~2.1]{HMSPic} when $p$ is
odd; in fact, the eigenvalues can be given by arbitrary $p$-adic
powers of $l$. \end{example}

Next, we need to review some facts about completion in the derived context,
and with an extra grading. 
Let $A_0$ be a commutative ring and let $I=(x_1,\ldots,x_t) \subseteq A_0$ be a finitely generated
ideal. In this case, one has the notion of an \emph{$I$-complete} object of the derived
category $D(A_0)$ (see \cite{DwGr}): an object $M$ in $D(A_0)$ is $I$-complete
if it is local for the object $A_0/x_1\otimes_{A_0}\cdots\otimes_{A_0} A_0/x_t$ (the iterated
cofiber, where all tensor products are derived). When $I=(x)$ is a principal
ideal, we will often write $x$-complete instead of $(x)$-complete.
Given a (discrete) $A_0$-module $M_0$, we will say that $M_0$ is \emph{derived
$I$-complete} if $M_0$, considered as an object of $D(A_0)$, is $I$-complete. 
Note the following fact about $I$-complete objects in the derived category. 

\begin{proposition}[{Compare \cite[Prop 5.3]{DwGr}}]
    Let $A_0$ be a commutative ring and let $I=(x_1,\ldots,x_t)$ be a finitely
    generated ideal. If $M \in D(A_0)$, then $M$ is $I$-complete if and only if 
    each homology group of $M$ is derived $I$-complete. \end{proposition} 

We will need an analog in the graded context. Let $A_*$ be a commutative,  
graded ring.
Let $I  = (x_1, \dots, x_t) \subseteq A_*$ be a finitely generated homogeneous
ideal generated by elements $x_1, \dots, x_t \in A_*$ of degrees $d_1,\dots,d_t$. 
Consider the derived category $D(A_*)^{\mathrm{gr}}$ of \emph{graded} $A_*$-modules. 
Let $(d)$ denote the operation of shifting the grading by $d$, so that  if $M \in D(A_*)^{\mathrm{gr}}$, then one has maps $x_i \colon  M_{(d_i)} \to M$.
\begin{definition} 
Given $M \in D(A_*)^{\mathrm{gr}}$, we say that $M$ is
\emph{$I$-complete} if the natural map $$M \to \lim_n \mathrm{cofib}( x_i^n
 \colon  M_{(nd_i)} \to M)$$ is an equivalence for $i = 1, 2, \dots , t$. Equivalently, for $i
= 1, 2, \dots, t$, one requires that the limit of multiplication by $x_i$ on
$M$ vanishes. 
We say that a (discrete) graded $A_*$-module  $M_*$ is 
\emph{derived $I$-complete} if it is $I$-complete when considered as an object of
$D(A_*)^{\mathrm{gr}}$. 
\end{definition} 

As in \cite[Prop. 5.3]{DwGr} (see also 
the treatment in \cite[Section~4]{DAGXII}), one shows that  
an object of $D(A_*)^{\mathrm{gr}}$ is $I$-complete if and only if the homology
groups (which are graded $A_*$-modules) are derived $I$-complete. 
This implies that the collection of derived $I$-complete graded $A_*$-modules is abelian and
contains every classically $I$-complete graded $A_*$-module. 

Let $A$ be an $\EE_\infty$-ring. Suppose that $\pi_*(A)$ is concentrated in
even degrees. Let $I  = (x_1, \dots, x_t) \subseteq \pi_*(A)$ be a finitely
generated ideal generated by elements $x_i \in \pi_{d_i}(A)$. Let $M$ be an $A$-module. 
We say that $M$ is \emph{$I$-complete} if $M$ is Bousfield local with respect
to the iterated cofiber $A/(x_1, \dots, x_t)$. 
Equivalently, $M$ is $I$-complete if and only if for $i = 1, \dots, t$, the
(homotopy) inverse limit of multiplication by $x_i$ on $M$ is null, i.e., 
$\lim ( \dots \stackrel{x_i}{\to} \Sigma^{d_i}M \stackrel{x_i}{\to} M) = 0$. 

\begin{proposition} 
\label{derivedorcomplete}
    If $A$ is an $\EE_\infty$-ring and $M$ is an $A$-module,
    then $M$ is $I$-complete if and only if $\pi_*(M)$ is
    derived $I$-complete as a graded $\pi_*(A)$-module.
\end{proposition} 
\begin{proof} 
In fact, for $i = 1, \dots, t$, the homotopy groups of $ \lim ( \dots
\stackrel{x_i}{\to} \Sigma^{d_i}M \stackrel{x_i}{\to} M)$ are computed by the Milnor short exact
sequence
\[ \lim{}^1  ( \dots
\stackrel{x_i}{\to} \pi_{* - d_i}M \stackrel{x_i}{\to} \pi_* M)
\hookrightarrow
\pi_* \left( \lim ( \dots
\stackrel{x_i}{\to} \Sigma^{d_i} M \stackrel{x_i}{\to} M) \right) 
\twoheadrightarrow
\lim  ( \dots
\stackrel{x_i}{\to} \pi_{* - d_i}M \stackrel{x_i}{\to} \pi_* M)
.
\]
If $M$ is $I$-complete, then  the middle term vanishes for each $i = 1, 2,
\dots, t$; thus the two outer terms vanish, which implies
precisely that $\pi_*(M)$ is derived $I$-complete. 
The converse follows similarly. 
\end{proof} 

\begin{lemma} 
\label{lem:finitelygen}
Let $A_*$ be  a commutative, graded ring which is derived $x$-complete for
some homogeneous element $x \in
A_d$. 
Let $M_*$ be a (discrete) graded $A_*$-module which is derived $x$-complete.
Suppose $M_*/(x) M_*$ is a finitely generated $A_*$-module. Then $M_*$ is a
finitely generated $A_*$-module. 
\end{lemma} 
\begin{proof} 
Let $F_*$ be a finitely generated graded free $A_*$-module together with a map $f \colon  F_*
\to M_*$ which induces a surjection $F_*/(x) F_* \to M_*/(x)M_*$. 
We claim that $f$ itself is a surjection. In fact, form the cofiber $Cf$ of $f$ in
the derived category $D(A_*)^{\mathrm{gr}}$; it suffices to see that $Cf$ has
homology concentrated in (homological, not graded) degree $1$. Note first that $\mathrm{cofib}(x \colon  Cf_{(d)} \to
Cf)$ is concentrated in degree $1$. But $Cf$ is $x$-complete, so $Cf \simeq
\lim_n \mathrm{cofib}(x^n \colon Cf_{(nd)}  \to Cf) $. The cofibers $
\mathrm{cofib}(x^n \colon Cf_{(nd)}  \to Cf)$ are concentrated in degree one by 
induction on $n$ and the transition maps are surjective; thus the homotopy
limit $Cf$ is concentrated in degree one by the Milnor exact sequence, as desired. 
\end{proof} 

\begin{proposition} 
\label{whenperfect}
Let $A$ be an $\EE_\infty$-ring such that $\pi_*(A)$ is a regular noetherian
ring of finite Krull dimension.
Then an $A$-module $M$ is perfect if and only if $\pi_*(M)$ is a finitely
generated $\pi_*(A)$-module.
\end{proposition} 
\begin{proof} 
We show that finite generation implies perfectness; the other direction follows from a thick subcategory argument. Note that the homological dimension of $\pi_*(M)$ as a $\pi_*(A)$-module is
 finite by regularity. We thus use induction on the homological
 dimension $\mathrm{h.dim}(\pi_*(M))$. If $\pi_*(M)$ is projective as a $\pi_*(A)$-module, 
 then it is also projective as a graded $\pi_*(A)$-module by a simple argument
 \cite[Cor. I.2.2]{gradedringtheory} and $M$ itself is a retract of a
 finitely generated free
 $A$-module. In general, choose a finitely generated free $A$-module $N$ with a
 map $N \to M$ inducing a surjection on $\pi_*$, and form a fiber sequence
 $N' \to N \to M$. We have that $\mathrm{hdim}(\pi_*(N')) =
 \mathrm{hdim}(\pi_*(M)) -1$. Since $M$ is perfect if and only if $N'$ is, this
 lets us conclude by the inductive hypothesis. 
\end{proof} 

\begin{proposition} 
Let $A$ be an $\EE_\infty$-ring such that $\pi_*(A)$ is a regular noetherian
ring of finite Krull dimension. Suppose that $x_1, \dots, x_t \in \pi_*(A)$ form
a regular sequence and $A$ is $(x_1, \dots, x_t)$-complete.\footnote{By
regularity, this is equivalent to the condition that $\pi_*(A)$ should be
classically $(x_1, \dots, x_t)$-complete.} Then: 
\begin{enumerate}
\item  

Let $M \in \Mod_A$ be an $(x_1, \dots, x_t)$-complete module such that
$M/(x_1, \dots, x_t)$ is perfect as an $A$-module. Then $M$ is perfect as an
$A$-module.

\item
Let $N$ be a dualizable object of the symmetric monoidal $\infty$-category of
$(x_1, \dots, x_t)$-complete $A$-modules. Then $N$ is perfect. 
\end{enumerate}
\label{descpf}
\end{proposition} 
\begin{proof} For (1), 
by Proposition~\ref{whenperfect},  it suffices to prove that $\pi_*(M)$ is finitely generated as an $\pi_*(A)$-module. 
By induction on $t$, one can reduce to the case where $t = 1$ where we let $x =
x_1$. 
In this case, $M/x$ is perfect as an $A$-module by assumption. 
It follows that $\pi_*(M)/(x) \pi_*(M) \subset \pi_*(M/x)$ is finitely generated as a
$\pi_*(A)$-module. 
Note that $\pi_*(M)$ is a derived $x$-complete $\pi_*(A)$-module (Proposition~\ref{derivedorcomplete}). 
By Lemma~\ref{lem:finitelygen}, it follows that $\pi_*(M)$ is a finitely generated
$\pi_*(A)$-module, completing the proof of (1).

For (2), we will argue that $N/(x_1, \dots, x_t)$ is perfect as an $A$-module
so that we can apply (1).
Let $\widehat{\Mod_A}$ denote the symmetric monoidal $\infty$-category
of complete $A$-modules. We have internal mapping objects in $\widehat{\Mod_A}$:
namely, given $N_1, N_2 \in \widehat{\Mod_A}$, the usual mapping
$A$-module $\Hom_A(N_1, N_2)$ is also
complete and yields the internal mapping object in $\widehat{\Mod_A}$. 
Since $N$ is dualizable in complete $A$-modules, it has a dual $\mathbb{D}N \in
\widehat{\Mod_A}$ and 
for any $N_2$, we have an equivalence
\[ \Hom_A(N, N_2) \simeq  \mathbb{D}N \widehat{\otimes}_A N_2,\]
where $\widehat{\otimes}_A$ denotes the completed tensor product. 
Tensoring with $A/(x_1, \dots, x_t)$, it follows that for any $N_2 \in
\widehat{\Mod_A}$, 
\begin{equation}  \label{A0} \Hom_A(N, N_2) /(x_1, \dots, x_t) \simeq \mathbb{D} N \otimes_A N_2/(x_1,
\dots, x_n).  \end{equation}
Note that both sides of \eqref{A0} make sense for $N_2 \in \Mod_A$, and are
unchanged by passage to the completion: therefore, \eqref{A0} holds for arbitrary
$N_2 \in \Mod_A$.
The left-hand-side is identified with a shift of $\Hom_A( N/(x_1, \dots, x_t),
N_2)$ and the right-hand-side commutes with filtered colimits; thus, unwinding the definitions, we find that $N/(x_1, \dots, x_t) \in
\Mod_A$ is a compact  object and therefore perfect. By (1), this shows that $N$
is perfect as an $A$-module, as desired. 
\end{proof} 

We can now state and prove our main theorem about perfectness.

\begin{theorem} 
\label{technicalperfect}
Let $G$ be a compact, connected Lie group and let $R \in \CAlg( \Sp^{BG})$. Suppose that
$\pi_*(R)$ is a regular noetherian ring of finite Krull dimension concentrated
in even degrees and that
$\H^*( BG; \pi_0 R)$ is a polynomial algebra over $\pi_0 R$ on even-dimensional classes.  Then any dualizable
object in  $\Mod_{R}( \Sp^{BG})$ is perfect. 
\end{theorem} 
\begin{proof}
We first show that in the above situation, the graded ring $\pi_* (R^{hG})$ is a regular
noetherian ring concentrated in even degrees.\footnote{In the example of
interest below, one can simply compute the ring explicitly and thus the
first two paragraphs become redundant.} 
The fact that $\pi_*(R^{hG})$ is concentrated in even degrees follows from the homotopy
fixed point spectral sequence, which degenerates for degree reasons. 
By Corollary~\ref{cor:nontrivunipotence}, 
one has classes $x_1, \dots, x_t \in \pi_*(R^{hG})$ such that the $x_i$'s
form a regular sequence in $\pi_*(R^{hG})$, the $x_i$'s
map to zero under $\pi_*(R^{hG}) \to \pi_*(R)$, and such that 
one has equivalences
\[ R^{hG}/(x_1, \dots, x_t) \simeq R, \quad \pi_*(R^{hG})/(x_1, \dots, x_t) \simeq \pi_*(R).  \]
The $(x_1, \dots,x_t)$-adic filtration on $\pi_*(R^{hG})$ is complete by the
convergence of the spectral sequence and it has noetherian associated graded, so $\pi_*(R^{hG})$ is noetherian at least as a graded ring.
Therefore, $\pi_*(R^{hG})$ is noetherian as an ungraded ring as well \cite[Th. 1.1]{GY}. 

By completeness, it follows that any maximal graded ideal of $\pi_* (R^{hG})$
contains $(x_1, \dots, x_t)$. 
Given a maximal graded ideal $\mathfrak{m} \in \pi_* (R^{hG})$, it follows that 
the (ungraded) localization 
$\pi_* (R^{hG})_{\mathfrak{m}}$ 
is a regular local ring: in fact, the quotient 
$\pi_* (R^{hG})_{\mathfrak{m}}/(x_1, \dots, x_t)$ by the regular sequence $x_1,
\dots, x_t$ is a regular local ring since it is a local ring of $\pi_*(R)$. By \cite[Th. 2.1]{3local}, this implies that $\pi_*(R^{hG})$ itself is a
regular ring.
Finally, we argue that $\pi_*(R^{hG})$ has finite Krull dimension. Let
$\mathfrak{p} \subseteq \pi_* (R^{hG})$  be a prime ideal of height $h$; up to
replacing $h$  by $h-1$, we can assume that $\mathfrak{p}$ is homogeneous
by \cite[Prop. 1.3]{3local}. 
Up to enlarging $\mathfrak{p}$, we can then assume $(x_1, \dots, x_t) \in
\mathfrak{p}$. But then $\dim \pi_* (R^{hG})_{\mathfrak{p}} \leq 
\dim \pi_*(R)_{\overline{\mathfrak{p}}} + t$ for $\overline{\mathfrak{p}}$ the
image of $\mathfrak{p}$ in $\pi_*(R)$. This shows that $h$ is globally bounded
and thus that $\pi_*(R^{hG})$ has finite Krull dimension. 

Recall now that $\Mod_R(\Sp^{BG})$ is unipotent by 
Corollary~\ref{cor:nontrivunipotence}, and is identified as a symmetric
monoidal $\infty$-category with $(x_1, \dots, x_t)$-complete $R^{hG}$-modules. 
Let $M \in \Mod_R( \Sp^{BG})$ be dualizable. Unwinding the equivalence, it follows that $M^{hG}$ is dualizable in the
$\infty$-category $\widehat{
\Mod_{R^{hG}}}$ of complete $R^{hG}$-modules. 
But by part (2)
Proposition~\ref{descpf}, it follows that $M^{hG}$ is perfect as an $R^{hG}$-module. 
Thus, by Proposition~\ref{prop:unipotencecriterion}, $M$ is perfect in $\Mod_{R}(\Sp^{BG})$ as desired. 
\end{proof}

\section{Applications to $\THH$}

In this section, we give our main applications to $\THH$ and $\TP$. After some
preliminary remarks we prove our main theorem (Theorem~\ref{perfres2}) and give the application to
K\"unneth theorems (Theorem~\ref{ourKunneth}). These results
immediately imply Theorem~\ref{BMthm} of Blumberg and Mandell. Then, we give a slightly
different, second proof of the K\"unneth
theorem based on the observation that $\TP$ is an integral lift of $\HP$  (Theorem~\ref{integrallift}).
Finally, we go on to discuss a similar result in mixed characteristic.

\newcommand{\catk}{\mathrm{Cat}_{\infty, k}^{\mathrm{perf}}}
Let $k$ be a perfect field of positive characteristic. 
One needs the following fundamental calculation of B\"okstedt (cf.
\cite[Section~5]{HM97}) of the homotopy ring of $\THH(k)$. 

\begin{theorem}[B\"okstedt]
\label{bokstedt} There is an isomorphism of graded rings
    $\pi_*\THH(k) \simeq k[\sigma]$ where $|\sigma| = 2$.
\end{theorem}

Here $\THH(k)$ is naturally an $\EE_\infty$-ring spectrum equipped with an
$S^1$-action. One can thus consider modules in $\Sp^{BS^1}$ over $\THH(k)$.
One has $\pi_*(\THH(k)^{hS^1}) \simeq W(k)[x, \sigma]/(x \sigma -
p)$, a regular noetherian ring\footnote{We refer to \cite[Section~IV-4]{nikolaus-scholze}
for a treatment of this calculation.}. Here $\sigma$ is a lift of the B\"okstedt element and $x \in \pi_{-2} (\THH(k)^{hS^1})$ is a generator which is detected in filtration two in the homotopy fixed point spectral sequence. 
Using Theorem~\ref{technicalperfect}, one obtains the following result,
stated partially as Theorem~\ref{perfectresult} of the introduction. 

\begin{theorem} \label{perfres2} Let $k$ be a perfect field of characteristic
    $p>0$. There is an equivalence of symmetric monoidal
    $\infty$-categories between $\Mod_{\THH(k)}(\Sp^{BS^1})$ and $x$-complete $\THH(k)^{hS^1}$-modules. 
    Moreover any dualizable object in $\Mod_{\THH(k)}(\Sp^{BS^1})$ is perfect.
\end{theorem} 

\begin{proof}
    With B\"okstedt's calculation in hand, this result is a special case of
	 Corollary~\ref{cor:nontrivunipotence} and 
    Theorem~\ref{technicalperfect}.
\end{proof}

We now give the application to the K\"unneth formula. 
We will work with small, idempotent-complete $k$-linear stable
$\infty$-categories (which can be modeled as dg categories). 
These are naturally organized into an $\infty$-category $\catk$.
For any such $\mathcal{C}$, we can consider the topological Hochschild homology
$\THH(\mathcal{C})$, together with its natural $S^1$-action. 
We recall that $\THH$ defines a symmetric monoidal functor 
of $\infty$-categories
\[ \catk  \to \Mod_{\THH(k)}(\Sp^{BS^1}). \]
Compare the discussion in \cite[Section~6]{BGTmult} for stable $\infty$-categories, and we refer to \cite[Th. 14.1]{blumberg-mandell-tp} for a proof for $k$-linear
$\infty$-categories (at least those with a compact generator, to which the general result reduces). 

Recall (cf. \cite[Prop. 1.5]{ToenderAz}, \cite[Th. 3.7]{BGT})
that the dualizable objects in $\catk$ are precisely the smooth and proper
$k$-linear stable $\infty$-categories.

\begin{theorem} 
    \label{ourKunneth}
    Let $k$ be a perfect field of characteristic $p>0$.
    Let $\mathcal{C}, \mathcal{D}$ be $k$-linear dg categories and suppose
    $\mathcal{C}$ is smooth and proper. 
    Then for any closed subgroup $H \subseteq S^1$, $\THH(\mathcal{C})^{tH}$ is a perfect
    $\THH(k)^{tH}$-module and the natural map 
    \[\THH(\mathcal{C})^{tH} \otimes_{\THH(k)^{tH}}\THH(\mathcal{D})^{tH} \to
    \THH(\mathcal{C} \otimes_k \mathcal{D})^{tH}  \]
    is an equivalence.  The same holds with $tH$ replaced by $hH$.
\end{theorem}

\begin{proof} 
    Since the functor 
    $\Mod_{\THH(k)}(\Sp^{BS^1}) \to 
    \Mod_{\THH(k)^{tH}}, X \mapsto X^{tH}$ is lax symmetric monoidal and exact, it
    follows (as in the discussion in the introduction after Theorem~\ref{BMthm}) that 
    it suffices to prove that $\THH(\mathcal{C}) \in \Mod_{\THH(k)}(\Sp^{BS^1})$ is
    perfect. 
    However, the construction $\mathcal{C} \mapsto\THH(\mathcal{C})$ is symmetric
    monoidal, and $\mathcal{C}$ is dualizable in $\catk$.  Since symmetric monoidal functors
    preserve dualizable objects, it follows from
    Theorem~\ref{perfres2} that $\THH(\mathcal{C}) \in
    \Mod_{\THH(k)}(\Sp^{BS^1})$ is perfect, completing the proof. 
\end{proof}

As a complement, we give a slightly different proof of Theorem~\ref{BMthm} based on the
result (Theorem~\ref{integrallift}) that $\TP$ is an integral lift of periodic cyclic homology $\HP(\cdot/k) $. 
This result also 
appears (in the more general setting of perfectoid rings) in the work of Bhatt-Morrow-Scholze \cite{BMS2} on integral $p$-adic Hodge
theory, and plays a role in various key steps there. We include a
proof for the convenience of the reader.

Note to begin that one has an $S^1$-equivariant map of $\EE_\infty$-rings
\[ HW(k) \to\THH(k)  \]
inducing an equivalence $(HW(k))^{tS^1} \simeq\TP(k)$.
When $k  = \mathbb{F}_p$, the map arises from the cyclotomic trace $\widehat{K(
\mathbb{F}_p)}_p \simeq H\mathbb{Z}_p \to\THH( \mathbb{F}_p)$ (the first
equivalence by Quillen
\cite{Qui}); it can also be constructed by computing $\TC( \mathbb{F}_p)$
(\cite{HM97}, \cite[Section~IV-4]{nikolaus-scholze}).  
In general, the same argument gives a map $H\mathbb{Z}_p \to \THH(k)^{hS^1}$ and the calculation of $\pi_0(\THH(k)^{hS^1}) = W(k)$ yields an extension on $\pi_0$. The 
existence of the spectrum level extension to a map from $HW(k)$ follows by obstruction theory from the $p$-adic
vanishing of the cotangent complex of $W(k)$ over $\mathbb{Z}_p$. Since the map
$(HW(k))^{tS^1}\rightarrow\TP(k)$ induces a $W(k)$-linear map on homotopy rings $W(k)[t^{\pm
1}]\rightarrow W(k)[x^{\pm 1}]$, we see it is an equivalence. 

\begin{theorem}[{Compare \cite[Theorem 6.7]{BMS2}}] \label{integrallift}
    Let $k$ be a perfect field of characteristic $p>0$.
    For any $\mathcal{C} \in \catk$, the natural map $\TP(\mathcal{C})
    \otimes_{\TP(k)} Hk^{tS^1} \simeq\TP(\mathcal{C})/p
    \to
    \HP(\mathcal{C}/k)$ is an equivalence. 
\end{theorem} 

\begin{proof} 
In fact, one has an equivalence in $\Sp^{BS^1}$
\[\THH(\mathcal{C}) \otimes_{\THH(k)} Hk \simeq\HH(\mathcal{C}/k), \quad
\mathcal{C} \in \catk.   \]
If $\mathcal{C} = \Perf(A)$ for an $\EE_1$-algebra, then the equivalence arises
from the 
equivalence of cyclic objects
\[ N^{\mathrm{cy}}(A) \otimes_{N^{\mathrm{cy}}(k)}  Hk \simeq N^{\mathrm{cy},
k}(A),  \]
where $N^{\mathrm{cy}}$ denotes the cyclic bar construction in 
spectra and $N^{\mathrm{cy}, k}$ in $Hk$-module spectra. 

In $\Mod_{\THH(k)}(\Sp^{BS^1})$,  we observe that $Hk$ is perfect as it
is the cofiber of the map $\Sigma^2\THH(k)\to\THH(k)$ given by multiplication
by $\sigma$, which can be made $S^1$-equivariant by the degeneration of the
homotopy fixed point spectral sequence. It follows thus that
\[\HH(\mathcal{C}/k)^{tS^1} \simeq \left(\THH(\mathcal{C}) \otimes_{\THH(k)}
Hk\right)
^{tS^1} \simeq\TP(\mathcal{C}) \otimes_{\TP(k)} Hk^{tS^1}, \]
as desired. 
\end{proof} 

\begin{proof}[Second proof  of Theorem~\ref{BMthm}]
Let $\mathcal{C}$ be smooth and proper over $k$. 
Since $\THH(\mathcal{C})$ is bounded below and $p$-torsion, $\TP(\mathcal{C})$ is
automatically $p$-complete. To see that $\TP(\mathcal{C})$ is a perfect
$\TP(k)$-module, 
it suffices to show that the homotopy groups of $\TP(\mathcal{C})$ are
finitely generated $W(k)$-modules (Proposition~\ref{whenperfect}). 
However, the homotopy groups $\pi_* (\TP(\mathcal{C})/p )\simeq
\HP_*(\mathcal{C}/k)$ are finite-dimensional $k$-vector spaces, which forces the homotopy groups
of $\TP(\mathcal{C})$ to be finitely generated $W(k)$-modules by
Lemma~\ref{lem:finitelygen}. (Compare with
the argument in the proof of~\cite[Theorem~16.1]{blumberg-mandell-tp}.)

Similarly, if $\mathcal{D}$ is another dg category over $k$ such that $\THH(\mathcal{D})$ is bounded below (e.g., if $\mathcal{D}$ is smooth and proper)
then $\TP(\mathcal{D})$, $\TP(\mathcal{C}\otimes_k\mathcal{D})$, and the tensor product $\TP(\mathcal{C}) \otimes_{\TP(k)}\TP(\mathcal{D})$ are
automatically $p$-complete already. To prove the K\"unneth formula, it thus
suffices to base-change along $\TP(k) \to Hk^{tS^1}$, since reduction modulo
$p$ is conservative for $p$-complete objects. But one already has a
K\"unneth formula in $\HP(\cdot/k)$,\footnote{For instance, this follows in a
similar  (but easier) fashion as dualizable objects in $\fun(BS^1, \Mod_k)$ are perfect in
view of the Postnikov filtration.} so one concludes. 
\end{proof} 

One also has a variant of B\"okstedt's calculation in mixed characteristic. 
Let $A$ be the localization of the ring of integers in a number field at a prime ideal lying over $p$.  
Fix a uniformizer $\pi \in A$. 
Let $S^0[q]$ denote the $\EE_\infty$-ring $\Sigma^\infty_+ \mathbb{Z}_{\geq 0}$. 
We consider the map of $\EE_\infty$-rings $S^0[q] \to HA$ sending $q \mapsto
\pi$.
The following result has now also appeared as~\cite[Proposition~11.10]{BMS2}. We are grateful to Lars Hesselholt for explaining it to us. 
For the reader's convenience, we include a proof. 

\begin{theorem} 
\label{mixcharBok}
    Let $A$ be the localization of the ring of integers in a number field at a
    prime ideal $(\pi)$ lying over $p$ and consider $HA$ as an $S^0[q]$-algebra as above.
    There is an isomorphism of graded rings $\pi_*\THH(A/S^0[q]) \simeq A[\sigma]$ where $|\sigma| = 2$. 
\end{theorem} 
\begin{proof}[Proof sketch] 
    Note first that the Hochschild homology groups $\HH_*(A/\mathbb{Z}[q])$ are finitely generated $A$-modules in each degree as  
    $A$ is a finitely generated module over $A \otimes_{\mathbb{Z}[q]} A$,
    which in turn  is a noetherian ring. Since  $\THH(A/S^0[q])
    \otimes_{\THH(\mathbb{Z}[q]/S^0[q])} H\mathbb{Z}[q] \simeq
    \HH(A/\mathbb{Z}[q])$, 
	and the homotopy groups of $\THH(\mathbb{Z}[q]/S^0[q])$ are finitely
	generated $\mathbb{Z}[q]$-modules, it follows easily that
    the homotopy groups of $\THH(A/S^0[q]) $ are finitely generated $A$-modules.
	Here we use repeatedly the following observation: if $R$ is a connective
	$\mathbb{E}_\infty$-ring with $\pi_0(R)$ noetherian and $\pi_i(R)$ finitely
	generated over $\pi_0(R)$, and $M  \in \Mod_R$ connective, the statement
	that the homotopy groups $\pi_j(M \otimes_R  H \pi_0 R)$ are finitely
	generated $\pi_0(R)$-modules implies that the homotopy groups $\pi_j(M)$ are finitely
	generated $\pi_0(R)$-modules. 
	 
	Now we have 
    $\THH(A/S^0[q])\otimes_{S^0[q]}S^0\we\THH(A/(\pi))$ (where $S^0[q] \to S^0$
	sends $q \mapsto 0$).
    By B\"okstedt's calculation in
    Theorem~\ref{bokstedt}, this is also a polynomial ring on a degree two
	 class. 
    It follows that the homotopy groups $\pi_*( \THH(A/S^0[q])$
    (which are finitely generated $A$-modules) are $\pi$-torsion-free and hence
    free, and 
    become a polynomial ring on a class in degree two after taking the quotient by
    $\pi$. 
    Letting $\sigma \in \pi_2( \THH(A/S^0[q])$ be a generator, we now obtain the
    result because $A$ is local.  
\end{proof}

One can use Theorem \ref{mixcharBok} to calculate $\TP(A/S^0[q])$ as in the case of perfect
fields of characteristic $p>0$. 
For example, $\pi_* \TP( \mathbb{Z}_{(p)}/S^0[q]) \simeq
\widehat{\mathbb{Z}_{(p)}[q]}_{(q-p)}[x^{\pm 1}]$ for $|x| = 2$. 
Using regularity as before, one obtains the following result from
Theorem~\ref{technicalperfect}.

\begin{corollary}\label{cor:sq} 
    Let $A$ be the localization of the ring of integers in a number field at a
    prime ideal lying over $p$ and consider $HA$ as an $S^0[q]$-algebra as above.
    In the symmetric monoidal $\infty$-category $\Mod_{\THH(A/S^0[q])}( \Sp^{BS^1})$,
    any dualizable object is perfect. 
\end{corollary} 

\newcommand{\cata}{\mathrm{Cat}_{\infty, A}^{\mathrm{perf}}}

Let $\cata$ denote the $\infty$-category of small, $A$-linear stable
$\infty$-categories. One has a symmetric monoidal functor
\[\THH(\cdot/S^0[q])  \colon  \cata \to \Mod_{\THH(A/S^0[q])}( \Sp^{BS^1})  \]
and one may define $\TP(\cdot/S^0[q]) \stackrel{\mathrm{def}}{=}
\THH(\cdot/S^0[q])^{tS^1}$.  Using arguments as above, one obtains the following. 

\begin{corollary} 
    Let $A$ be the localization of the ring of integers in a number field at a
    prime ideal lying over $p$ and consider $HA$ as an $S^0[q]$-algebra as above.
    Let $\mathcal{C}, \mathcal{D}$ be $A$-linear stable $\infty$-categories and suppose
    $\mathcal{C}$ is smooth and proper. Then $\TP(\mathcal{C}/S^0[q])$ is a
    perfect $\TP(A/S^0[q])$-module and the natural map
    \[\TP_{}(\mathcal{C}/S^0[q]) \otimes_{\TP(A/S^0[q])}\TP(\mathcal{D}/S^0[q]) \to
    \TP( \mathcal{C}\otimes_A \mathcal{D}/S^0[q])\] is an equivalence.
\end{corollary} 


If $\pi \in A$ is a uniformizer, one also has a map 
\[ S^0 [q^{\pm 1}] \to HA, \quad q \mapsto 1 + \pi.  \]
One can carry out a slight variant of the above calculations for $\THH(\cdot/S^0[q^{\pm
1}])$ and 
$\TP(\cdot/S^0[q^{\pm 1}])$, and replace the base-change $q \mapsto 0$ with $q
\mapsto 1$. One obtains
\[ \pi_* \THH(A/S^0[q^{\pm 1}]) \simeq A[\sigma], \quad |\sigma| = 2.  \]

Suppose now that the base ring $A$ is given by $\mathbb{Z}_{(p)}[\zeta_p]$ and
$\pi = \zeta_p - 1$.
In this case, 
there is an isomorphism
\[ \pi_*\TP(\mathbb{Z}_{(p)}[\zeta_p]/S^0[q^{\pm 1}]) \simeq 
\widehat{\mathbb{Z}_{(p)}[q^{\pm 1}]}_{\Phi_p(q)}[x^{\pm 1}] . 
\]
Here $\Phi_p(q)$ is the $p$th cyclotomic polynomial. 
Moreover, one obtains a functor
\[ \mathrm{Cat}_{\infty, \mathbb{Z}_{(p)}[\zeta_p]}^{\mathrm{perf}} \to \Mod_{\TP(
\mathbb{Z}_{(p)}[\zeta_p]/S^0[q^{\pm 1}])},  \]
which analogs of our arguments show satisfies a K\"unneth formula for smooth and proper dg
categories over $\mathbb{Z}_{(p)}[\zeta_p]$.
Just as $\TP$ is analogous to $2$-periodic crystalline
cohomology, $\TP(\cdot/S^0[q^{\pm 1}])$ in this case is roughly analogous to a $2$-periodic version of the $q$-de Rham cohomology
proposed by Scholze \cite{scholze-qdef}. We refer to \cite{BMS2} for details. 


\section{A finiteness result}

We next apply our version of the K\"unneth theorem to prove the following finiteness
result for the topological cyclic homology $\TC(\mathcal{C})$ of a smooth and
proper dg category $\mathcal{C}$ over a \emph{finite} field. 

\begin{theorem} 
\label{TCfinite}
Suppose $k$ is a {finite} field of characteristic $p$. 
Let $\mathcal{C} \in \catk$ be smooth and proper. 
Then
$\TC(\mathcal{C})$ is a perfect $H\mathbb{Z}_p$-module.
\end{theorem} 

Of course, the above result fails for $\THH(\mathcal{C})^{hS^1}$, already for
$\mathcal{C} = \Perf(k)$. The finiteness for $\TC$ follows from the
Nikolaus--Scholze formula for $\TC$ in terms of $\THH^{hS^1}$ and $\TP$ and in
particular the interactions with the cyclotomic Frobenius. After gathering the
necessary facts about $\TC$ below and giving a lemma on natural transformations
of symmetric monoidal functors, we prove Theorem~\ref{TCfinite} at the end of
the section.

\begin{example} 
    Suppose $\mathcal{C} = \Perf(X)$ where $X$ is a smooth and projective variety
    over a finite field $k$.
    In this case, there is a stronger, more refined result of Geisser and Hesselholt \cite{GH}. They show 
    \cite[Prop. 5.1.1]{GH}
    that $\pi_i\TC(X)$ is finite for $i \neq 0, -1$. (It also follows from their
    descent spectral sequence and~\cite[Prop.~4.18]{gros-suwa} that
    $\pi_i\TC(X)$ is a finitely generated $W(k)$-module for $i=0,-1$.)
    Earlier computations of Hesselholt (see \cite[Theorem~B]{Hesselholt}) imply that the
    homotopy groups of $\TR(X)$ vanish in degrees $>\dim X$ and hence that the
    homotopy groups of $\TC(X)$ vanish in degrees $> \dim X$ since $\TC(X)$ is
    the spectrum of $F$-fixed points of $\TR(X)$;
    for descent-theoretic reasons \cite[Section~3]{GH} they vanish in degrees $< -\dim X -1$. 
    In addition, they show that $\TC(X)$ is identified with the $p$-adic \'etale
    $K$-theory of $X$ provided $X$ is smooth over $k$
    (see~\cite[Theorem~4.2.6]{GH}).
\end{example} 

\begin{definition} 
    In this section, we write $\TC^-(\mathcal{C})$ for $\THH(\mathcal{C})^{hS^1}$. 
\end{definition} 

We will need a number of preliminaries. 
In particular, we will use the Nikolaus--Scholze \cite{nikolaus-scholze} description of the
$\infty$-category $\CycSp$
of cyclotomic spectra. We will restrict to the $p$-local case.
\begin{definition}[{\cite[Def. II.1.6]{nikolaus-scholze}}]
\label{def:cycsp}
The homotopy theory
$\CycSp$ of  cyclotomic spectra is the presentably symmetric monoidal stable
$\infty$-category of pairs $$(X \in \Sp^{BS^1}, \{\varphi_p \colon  X \to
X^{tC_p}\}_{p \ \text{prime}})$$ where, for each prime $p$,
the map 
$\varphi_p$ is $S^1$-equivariant for 
the natural $S^1/C_p$-action on $X^{tC_p}$ and the natural identification $S^1
\simeq S^1/C_p$. 
For $X \in \CycSp_{}$, 
the \emph{topological cyclic homology} $\TC(X)$ is defined as the mapping
spectrum $\Hom_{\CycSp}(\mathbbm{1}, X)$ for $\mathbbm{1} \in \CycSp$ the unit. 
\end{definition} 

By \cite[Section~II.6]{nikolaus-scholze}, the above definition
agrees with earlier definitions of cyclotomic spectra (e.g., those in terms
of genuine equivariant stable homotopy theory) when the
underlying spectrum $X$ is bounded below. 
In the rest of the paper, $X$ will always be local at a fixed prime $p$. In
this case, the Tate constructions $X^{t C_q}$ for $q \neq p$ vanish, so that
the only relevant map is $\varphi \stackrel{\mathrm{def}}{=} \varphi_p$. 

Given
an object $X \in \CycSp$ such that $X$ is bounded below and in addition
$p$-complete, then $X^{hS^1}$ and $X^{tS^1}$ are also
$p$-complete\footnote{This follows by induction up the Postnikov tower.}
and one obtains a
Frobenius map $\varphi \colon  X^{hS^1} \to X^{tS^1} $, in addition to the canonical map
$\mathrm{can} \colon  X^{hS^1} \to X^{tS^1}$.  We will use the
fundamental formula \cite[Prop. II.1.9, Rmk. II.4.3]{nikolaus-scholze}
(valid for bounded below $p$-complete $X \in \CycSp$)
\begin{equation}
\TC(X) \simeq \mathrm{eq} \left( \mathrm{can}, \varphi \colon  X^{hS^1} \rightrightarrows X^{tS^1} \right).
\end{equation}

For a stable $\infty$-category $\mathcal{C}$, the topological Hochschild
homology construction $\THH(\mathcal{C})$ naturally admits the structure of a
cyclotomic spectrum (compare \cite{ayala-mg-rozenblyum}). For $k$-linear dg
categories one has a symmetric monoidal functor
\[\THH \colon \catk \to \Mod_{\THH(k)}(\CycSp ). \]
There is a related discussion of the symmetric
monoidality in~\cite[Section~6]{BGTmult} and we believe this is known to other authors as well.
For example, one can use the symmetric monoidality of the classical point-set
constructions and map that to the new $\infty$-category of cyclotomic spectra
to get the desired functor.

Suppose $\THH(\mathcal{C})$ is bounded below, e.g., if $\mathcal{C}$ is smooth
and proper over $k$ or if $\mathcal{C} = \Perf(A)$ for $A$ a connective
$\EE_1$-algebra in $k$-modules.
Then, in this language, the \emph{topological cyclic homology} $\TC(\mathcal{C})$ can
be defined as \begin{equation} \label{eqformula} \TC(\mathcal{C}) =\TC(\THH(\mathcal{C})) 
= \mathrm{eq} \left( \mathrm{can}, \varphi \colon \TC^-(\mathcal{C}) \rightrightarrows
\TP(\mathcal{C}) \right).\end{equation}

\begin{example} \label{ex:tcminus}
    We need two basic properties of the two maps $\mathrm{can}$ and $\varphi$. 
    Here we use more generally that for any perfect ring $\pi_* \TC^-(k) \simeq W(k)[x, \sigma]/(x \sigma - p)$ for $x \in
    \pi_{-2}$ and $\sigma \in \pi_2$. When $k=\FF_p$, this calculation follows
    from~\cite[Section~IV.4]{nikolaus-scholze}. The case of a general perfect
    ring follows by base-change using that if $k$ is perfect, the cotangent
    complex $\L_{k/\FF_p}\we 0$. See also~\cite[Section~6]{BMS2}.
    \begin{enumerate}
        \item The map $\mathrm{can} $ carries $x \in \pi_{-2}\TC^-(k)$ to an invertible
            element in $\TP(k)$ by~\cite[Section~IV.4]{nikolaus-scholze}; there it
            is proved for $k=\FF_p$ and the general case follows from the
            remarks above and naturality. For any
            $\mathcal{C}$, it follows that the map $\TC^-(\mathcal{C})[1/x]
            \to\TP(\mathcal{C})$ is an equivalence since $\TP(\Cscr)$ is
            obtained by the extension of scalars
            $\TP(k)\otimes_{\TC^-(k)}\TC^-(\Cscr)$
            by~\cite[Lemma~IV.4.12]{nikolaus-scholze}.
        \item The map $\varphi$ carries $\sigma \in \pi_2\TC^-(k)$ to an invertible element
            in $\TP(k)$. The map $\varphi$ induces an equivalence $\TC^-(k)[1/\sigma] \to\TP(k)$. 
            For any $\mathcal{C} \in \catk$ such that $\THH(\mathcal{C})$
            is bounded below, one has a map $\TC^-( \mathcal{C})[1/\sigma] \to
            \TP(\mathcal{C})$ which is $\varphi$-semilinear. 
            Alternatively, one has a map of $\TP(k)$-modules $\TC^-(\mathcal{C})
            \otimes_{\TC^-(k)}{_\varphi\TP(k)} \to\TP(\mathcal{C})$, where the map $\TC^-(k) \to\TP(k)$
            is $\varphi$.
    \end{enumerate}
\end{example} 

We now prove some facts about these invariants for smooth and proper dg
categories. First we need a preliminary proposition about dualizability.
\begin{proposition}\label{symmonprop} 
\begin{enumerate}
\item  
Let $\mathcal{T}$ be a symmetric monoidal $\infty$-category and let $\fun(
\Delta^1, \mathcal{T})$ denote the $\infty$-category of arrows in $\mathcal{T}$,
with the pointwise symmetric monoidal structure. Then any dualizable object
$f \colon  X_1 \to X_2$ of
$\fun(\Delta^1, \mathcal{T})$ has the property that the map $f$ is an
equivalence.
\item 
Let $\mathcal{T}, \mathcal{T}'$ be a symmetric monoidal $\infty$-categories. 
Let $F_1, F_2 \colon  \mathcal{T} \to \mathcal{T}'$ be symmetric monoidal functors and
let $t \colon  F_1 \to F_2$ be a symmetric monoidal natural transformation. Suppose 
every object of $\mathcal{T}$ is dualizable. Then $t$ is an equivalence. 
\end{enumerate}
\end{proposition} 
\begin{proof} The first assertion implies the second, so we focus on the first. 
Let $\vee$ denote duality on $\mathcal{T}$.
If $f \colon  X \to Y$ has a dual, then the source and target have to be the duals of
$X$ and $Y$ since the evaluation functors are symmetric monoidal. Thus, the
dual of $f \colon X\rightarrow Y$ is an arrow
$\bar{f} \colon  X^{\vee} \to Y^{\vee}$. We claim that $\bar{f}^{\vee}$ is the
inverse of $f$.

To see this, we draw some diagrams. We write $\ev, \coev$ for evaluation and coevaluation maps, respectively. Since $\bar{f}$ is the dual of $f$, one has a commutative triangle
$$
\xymatrix{
\mathbbm{1} \ar[d]^{\coev_X} \ar[rd]^{\coev_Y} \\
X \otimes X^{\vee} \ar[r]^{f \otimes \bar{f}} & Y \otimes Y^{\vee}.}
$$
As a result, it follows that the diagram
$$
\xymatrix{
Y^{\vee} \ar[d]^{\id \otimes \coev_X} \ar[r]^{\id} & Y^{\vee} \ar[dd]^{\id \otimes \coev_Y} \\
Y^{\vee} \otimes X\otimes X^{\vee} \ar[rd]^{\id\otimes f \otimes \bar{f}}     \ar[d]_{\id \otimes f \otimes \id} \\
Y^{\vee} \otimes Y \otimes X^{\vee} \ar[r]_{\id \otimes \id \otimes \bar{f}}
\ar[d]^{\ev_Y \otimes \id} &   Y^{\vee}\otimes Y \otimes Y^{\vee} \ar[d]^{\ev_Y \otimes \id} \\ 
X^{\vee} \ar[r]^{\bar{f}} & Y^{\vee}
}
$$
commutes.
Chasing both ways around the diagram, one finds that $\id_{Y^{\vee}} \we \bar{f} \circ f^{\vee}$. 
Dualizing again,  we get that $f \circ \bar{f}^{\vee}$ is equivalent to the
identity of $Y$. In particular, $f$ admits a section. To see that $f$ is
actually an equivalence, assume without loss of generality that all objects are dualizable.
Now apply the symmetric monoidal equivalence $\vee \colon  \fun(\Delta^1, \mathcal{T})
\simeq \fun(\Delta^1, \mathcal{T})^{op}$. It follows that $f^{\vee}$ admits a
section too. Therefore, $f$ is an equivalence.
\end{proof} 

\begin{proposition} \label{propequi}
\label{TCplus}
Let $k$ be a perfect field of characteristic $p>0$.
For $\mathcal{C} \in \catk$ smooth and proper, the $\varphi$-semilinear map 
\( \varphi \colon \TC^-(\mathcal{C})[1/\sigma] \to\TP(\mathcal{C})  \)
is an equivalence. Equivalently, one has an equivalence of $\TP(k)$-modules
\[\TC^-(\mathcal{C})\otimes_{\TC^-(k), \varphi} \TP(k) \simeq\TP(\mathcal{C}).  \]
\end{proposition} 
\begin{proof} 
Let $\mathcal{T}$ denote the $\infty$-category of smooth and proper objects in
$\catk$. Then both sides of the above displayed map yield symmetric monoidal functors to
$\Mod_{\TP(k)} $ in view of Theorem~\ref{ourKunneth}; for the right-hand-side
this is the Blumberg--Mandell theorem. The natural map is a symmetric monoidal
natural transformation, so the result follows from Proposition~\ref{symmonprop}.
\end{proof} 

\begin{proposition} 
    Let $k$ be a perfect field of characteristic $p>0$.
Let $\mathcal{C} \in \catk$ be smooth and proper.
\begin{enumerate}
\item  
For $i \gg 0$, the map $\sigma \colon \pi_i\TC^-(\mathcal{C}) \to \pi_{i+2}
\TC^-(\mathcal{C})$ is an isomorphism. 

\item
For $i \ll 0$, the map $x \colon  \pi_i\TC^-(\mathcal{C}) \to \pi_{i-2}
\TC^-(\mathcal{C})$ is an isomorphism. 

\item For $i \gg 0$, the map $\varphi \colon  \pi_i\TC^-(\mathcal{C}) \to \pi_i
\TP(\mathcal{C})$ is an isomorphism. 
\item 
 For $i \ll 0$, the map $\mathrm{can} \colon  \pi_i\TC^-(\mathcal{C}) \to \pi_i
\TP(\mathcal{C})$ is an isomorphism. 
\end{enumerate}
\end{proposition} 
\begin{proof} 
Assertions 1 and 2 follow from the fact that $\TC^-(\mathcal{C})$ is a perfect
$\TC^-(k)$-module by Theorem~\ref{ourKunneth} and the fact that they are true
for $\TC^-(k)$ (recalling Example~\ref{ex:tcminus}). Assertion 3 now follows from Assertion 1 and
Proposition~\ref{TCplus}. 
Assertion 4 follows from Assertion 2 and the fact that $\TP(\mathcal{C}) \simeq
\TC^-(\mathcal{C})[1/x]$ under $\mathrm{can}$.
\end{proof} 

\begin{proof}[Proof of Theorem~\ref{TCfinite}]
If $k$ is a finite field, then the $\mathbb{Z}_p$-modules $\pi_i\TC^-(\mathcal{C}),
\pi_i\TP(\mathcal{C})$ are finitely generated by Theorem~\ref{ourKunneth}.
Moreover, for $i \gg 0$, one finds that $\varphi$ is an isomorphism while
$\mathrm{can}$ is divisible by $p$ (as $\mathrm{can}(\sigma)$ is divisible by
$p$ in $\TP(k)$), while for $i \ll 0$, 
$\mathrm{can}$ is an isomorphism while $\varphi$ is divisible by $p$. The equalizer
formula 
\eqref{eqformula}
for $\TC$ yields the assertion. 
\end{proof} 
\section{A counterexample for cyclotomic spectra}\label{sec:examples}

Let $k$ be a perfect field of characteristic $p>0$ and let $\Cscr$ be a smooth
and proper dg category over $k$.
In view of the main result of the previous section and the fact that
$\THH(\Cscr)$ is dualizable in $\Mod_{\THH(k)}(\CycSp)$, one may now ask if
$\THH(\mathcal{C})$ is perfect in $\Mod_{\THH(k)}(\CycSp)$. 

We show that this fails. For a supersingular K3 surface $X$ over $k$, we show that
$\THH(X)  = \THH(\Perf(X))$ is not perfect in $\CycSp$. 
In fact, we show that 
$\TF(X)$ (see below)
is not compact as a $\TF(k)$-module spectrum. This will rely on basic facts about
crystalline cohomology and its description via the de Rham-Witt complex (in
particular for supersingular K3 surfaces), as well
as the connection between the fixed points of $\THH$ and the de Rham-Witt
complex.

\begin{definition} 
Recall that if $Y \in \CycSp$, then 
$Y$ in particular defines a genuine $C_{p^n}$-spectrum for each
$n$.\footnote{This was the basis of classical definitions of cyclotomic
spectra. If one follows instead 
Definition~\ref{def:cycsp}, we refer to the proof of Theorem II.4.10 in \cite{nikolaus-scholze} for a definition of the fixed points and to the proof of Theorem II.6.3 in loc. cit.  for a construction of the genuine spectrum.
 } 
 As a result, one can form the fixed points $\TR^n(Y) = Y^{C_{p^n}}$. 
The spectrum $\TF(Y)$ can be defined as the homotopy limit $\TF(Y) = \lim_n
\TR^n (Y)$
where the transition maps are the natural maps $F \colon  \TR^{n+1}(Y) \to \TR^n(Y)$ (inclusions of fixed points).
Thus $\TF$ defines an exact, lax symmetric monoidal functor
\[ \TF \colon  \CycSp \to \Sp.   \]
For a scheme $X$, we will write $\TF(X) = \TF( \Perf(X))$. 
\end{definition}

We refer to \cite{Hesselholt} for the basic structure theorems for
$\TR^{n}$ and $\TF$ of
smooth schemes over a perfect field $k$ and to \cite{HM97} for the
calculations over $k$ itself. In particular, the results of \cite{Hesselholt}
show that the homotopy groups of $\TF$  are closely related to the cohomology
groups of the de Rham--Witt complex \cite{illusie-derham-witt}. 
For a smooth algebra $A/k$, we let $W_n  \Omega_A^\bullet$ denote the level $n$ de
Rham-Witt complex of $A$ and $W \Omega_A^\bullet$ denote the de Rham--Witt
complex of $A$, so that $W \Omega_A^\bullet = \lim_R W_n
\Omega_A^\bullet$ where the maps in the diagram are the restriction maps $R$. 
Recall also that the tower of graded abelian groups $\left\{W_n
\Omega_A^\bullet\right\}$ is equipped with Frobenius maps $F \colon  W_{n+1}
\Omega_A^\bullet \to W_n \Omega_A^\bullet$. 

\begin{theorem}[{Hesselholt--Madsen \cite[Theorems 3.3 and 5.5]{HM97}}]
    Let $k$ be a perfect field of characteristic $p>0$.
\begin{enumerate}
\item 
One has $\pi_* \THH(k)^{C_{p^n}} \simeq W_{n+1}(k)[\sigma_n]$ for $|\sigma_n| = 2$.
The map $F \colon  \pi_* \THH(k)^{C_{p^{n+1}}} \to \pi_* \THH(k)^{C_{p^n}}$ is the Witt
vector Frobenius on $W_{n+1}(k)$ and sends $\sigma_{n+1}$ to $\sigma_n$.
\item One has 
$\TF_*(k)\we W(k)[\sigma]$ for $|\sigma|  = 2$. 

\end{enumerate}
\end{theorem} 

In particular, if $\THH(X)$ is perfect
in $\Mod_{\THH(k)}(\CycSp)$, then $\TF_i(X)$ is a finitely generated
$W(k)$-module for all $i$. We show that if $X$ is a supersingular K3 surface,
then $\TF_{-1}(X)$ is not finitely generated over $W(k)$.
Here we use the following fundamental result. 

\begin{theorem}[{Hesselholt \cite[Theorem B]{Hesselholt}}] 
\label{TRofsmooth}
Let $k$ be a perfect field of characteristic $p>0$. If $A$ is a smooth
$k$-algebra, then 
$\pi_*  \THH(A)^{C_{p^n}} \simeq W_{n+1} \Omega_A^{\bullet}
\otimes_{W_{n+1}(k)} W_{n+1}(k)[\sigma_{n}]$ with $|\sigma_n| = 2$.
The map $F \colon  \pi_*  \THH(A)^{C_{p^{n+1}}}\to  \pi_* \THH(A)^{C_{p^n}} $ 
is the tensor product of the map $F \colon  W_{n+1}  \Omega_A^\bullet\to W_n
\Omega_A^\bullet$ and the map $W_{n+1}(k)[\sigma_{n+1}] \to W_{n}(k)[\sigma_n]$
acting as the Witt vector Frobenius $F \colon W_{n+1}(k)\rightarrow W_n(k)$ and
sending $\sigma_{n+1}$ to $\sigma_n$. 
\end{theorem} 


Now let $X$ be a smooth and proper scheme over $k$. 
Note that the construction $U \mapsto \TR^n(U)$ is a sheaf of spectra
$\underline{\TR^n}$ on the small \'etale site of $X$ (cf. \cite[Section~3]{GH}).
By Theorem~\ref{TRofsmooth}, 
there is a strongly convergent descent spectral sequence
$$\E^2_{s,t}=  \H^{-s}(X, \pi_t \underline{\TR^n})
\we
\H^{-s}(X,\bigoplus_{j=0}^\infty 
W_{n+1}\Omega_X^{t-2j})\Rightarrow\TR^n_{s+t}(X),$$ and taking the limit over $n$
we obtain a strongly convergent spectral sequence
$$\E^2_{s,t}=\lim_{n,F}\H^{-s}(X,\bigoplus_{j=0}^{\infty}W_n\Omega_X^{t-2j})\Rightarrow\TF_{s+t}(X).$$
Note that since $X$ is proper, all of the terms at each stage of the inverse limit are finite length, so no $\lim^1$ terms appear.
Note also that this spectral sequence comes from filtering
the \'etale sheaf $\underline{\TF}$ given by $U \mapsto \TF(U)$ via the tower
obtained by taking the inverse limit of the Postnikov towers of the
$\underline{\TR^n}$ (and not the Postnikov tower of $\underline{\TF}$ itself).
We will show using this spectral sequence that $\TF_{-1}(X)$ is
non-finitely generated when $X$ is a supersingular K3 surface.

As discussed in~\cite[Section~5]{hesselholt-tp}, the terms in the $\E^2$-page
of the spectral sequence computing $\TF$ above appear in the conjugate spectral
sequence
$$\E_2^{s,t}=\lim_{n,F}\H^s(X,W_n\Omega^t_X)\Rightarrow\H^{s+t}_{\mathrm{crys}}(X/W)$$
computing crystalline cohomology. This spectral sequence is studied in~\cite{illusie-raynaud} together
with the Hodge spectral sequence
$$\E_2^{s,t}=\H^s(X,W\Omega^t_X)\Rightarrow\H^{s+t}_{\mathrm{crys}}(X/W)$$
arising from the naive filtration of the de Rham--Witt complex. When $X$ is a
surface, $\lim_{n,F}\H^s(X,W_n\Omega^t_X)=0$ and $\H^s(X,W\Omega^t_X)=0$ for
$s>2$ or $t>2$. Moreover, the crystalline cohomology of a smooth proper
$k$-scheme is finitely generated over $W(k)$.


\begin{proposition}
The term $\lim_{n, F} \H^2(X,W_n\Omega^1_X )$ is not finitely generated if $X$ is a supersingular K3 surface. 
\end{proposition}
\begin{proof}

Recall from Figure~\ref{fig:k3cohomology} the Hodge--Witt cohomology groups $\H^i(X, W \Omega^j_X)$ for $X$ a
supersingular K3 surface taken from~\cite[Section~7.2]{illusie-derham-witt}.
\begin{figure}[h]
    \centering
    \begin{tabular}{c c c c}
        &   $W\Omega_X^0$   &   $W\Omega^1_X$   &   $W\Omega_X^2$\\[0.5ex]
        \hline
        $\H^2$  &   $k[[x]]$ & $k[[y]]$ & $W(k)$\\[0.5ex]
        $\H^1$  &   $0$ &   $W(k)^{22}$    & $0$\\[0.5ex]
        $\H^0$  &   $W(k)$ &   $0$ &   $0$
    \end{tabular}
    \caption{The Hodge--Witt cohomology of a supersingular K3 surface.}
    \label{fig:k3cohomology}
\end{figure}
By results of Illusie--Raynaud, the non-finite generation of
$\H^2(X,W\Omega_X^0)$ and $\H^2(X,W\Omega^1_X)$ over $W(k)$ has consequences in the
conjugate spectral sequence. We use the following general facts proved
in~\cite{illusie-raynaud}.
\begin{enumerate}
    \item[(1)] If
        $\H^{s}(X,W\Omega^t_X)$ is finitely generated for all $s+t=n$ (one says
        that $X$ is Hodge--Witt in degree $n$), then
        $\lim_{n,F}\H^s(X,W_n\Omega^t_X)$ is torsion-free for all $s+t=n$ if and only if 
        $\H^s(X,W\Omega^t_X)$ is torsion-free for all
        $s+t=n$~\cite[IV.4.6.1]{illusie-raynaud}.
    \item[(2)] Modulo torsion, the terms in the conjugate spectral sequence are
        finitely generated~\cite[Introduction, 2.2]{illusie-raynaud}.
    \item[(3)] If $\H^s(X,W\Omega^t_X)$ is not finitely generated for some pair
        $(s,t)$,
        then there is a pair $(s',t')$ such that
        $\lim_{n,F}\H^{s'}(X,W_n\Omega^{t'}_X)$ is not finitely
        generated~\cite[Introduction, 2.3]{illusie-raynaud}.
\end{enumerate}

From Figure~\ref{fig:k3cohomology} and points (1) and (2) above, we see that 
the terms $\lim_{n,F}\H^s(X,W_n\Omega^t_X)$ of total degrees $0,1,4$ in the conjugate spectral sequence are
finitely generated torsion-free $W(k)$-modules. From (3), we see that some
terms of total degrees $2$ and $3$ must be non-finitely generated.

We need to identify exactly where the failure of finite generation occurs. The
conjugate spectral sequence degenerates after the $\E_2$-page for dimension
reasons. Thus, the only possibly non-zero differentials are
$$\lim_{n,F}\H^0(X,W_n\Omega_X^2)\rightarrow\lim_{n,F}\H^2(X,W_n\Omega^1_X)$$
and
$$\lim_{n,F}\H^0(X,W_n\Omega_X^1)\rightarrow\lim_{n,F}\H^2(X,W_n\Omega_X^0).$$
The term $\lim_{n,F}\H^0(X,W_n\Omega_X^1)$ is in total degree $1$ and hence is
torsion-free and finitely generated. It follows that
$\lim_{n,F}\H^2(X,W_n\Omega_X^0)$ is finitely generated as well since otherwise
it would contribute something infinitely generated in crystalline cohomology.
Therefore, both
$$\lim_{n,F}\H^0(X,W_n\Omega_X^2), \quad \lim_{n,F}\H^2(X,W_n\Omega^1_X)$$
are non-finite over $W(k)$. (The kernel and cokernel of the differential are,
however, finite.) We see in particular that $\lim_{n,F}\H^2(X,W_n\Omega^1_X)$ is
non-finite over $W(k)$.
\end{proof}

We can now state and prove the main conclusion of this section. 

\begin{corollary}
    Let $k$ be a perfect field of characteristic $p>0$.
    If $X$ is a supersingular K3 surface over $k$, then $\TF_{-1}(X)$ is not
    finitely generated as a $W(k)$-module. In particular, $\THH(X) \in
    \Mod_{\THH(k)}(\CycSp)$ is not perfect. 
\end{corollary}
\begin{proof}

    Returning to the local-global spectral sequence for $\TF(X)$,
    Figure~\ref{fig:tfssk3} displays the region of the spectral sequence of
    interest to us.
\begin{figure}[h]
    \begin{equation*}
        \xymatrix{
        \lim_{n,F}\H^2(X,W_n\Omega^1)   &   &\\
        \lim_{n,F}\H^2(X,W_n\Omega^0\oplus W_n\Omega^2)
        &\lim_{n,F}\H^1(X,W_n\Omega^0\oplus
    W_n\Omega^2)&\lim_{n,F}\H^0(X,W_n\Omega^0\oplus W_n\Omega^2)\ar[ull]^{d^2}\\
        \lim_{n,F}\H^2(X,W_n\Omega^1)&\lim_{n,F}\H^1(X,W_n\Omega^1)&\lim_{n,F}\H^0(X,W_n\Omega^1)\ar[ull]^{d^2}\\
        \lim_{n,F}\H^2(X,W_n\Omega^0)&\lim_{n,F}\H^1(X,W_n\Omega^0)&\lim_{n,F}\H^0(X,W_n\Omega^0)\ar[ull]^{d^2}\\
        }
    \end{equation*}
    \caption{A part of the local-global spectral sequence for $\TF$ of a
    surface. To fix coordinates, the bottom left term displayed is $\E^2_{-2,0}$.}
    \label{fig:tfssk3}
\end{figure}
All terms and differentials contributing to $\TF_{s+t}(X)$ for $s+t\leq 1$ are shown.
However, $\lim_{n,F}\H^0(X,W_n\Omega^0_X)$ is torsion-free and finite over
$W(k)$ (since it has degree $0$ in the conjugate spectral sequence for
crystalline cohomology). Thus, the $d^2$-differential hitting
$\lim_{n,F}\H^2(X,W_n\Omega^1)$ cannot possibly annihilate enough for the
resulting term $\E^3_{-2,1}\iso\E^\infty_{-2,1}$ to be a finite $W(k)$-module.
Of course, this implies that $\TF_{-1}(X)$ is not finitely generated, which is
what we wanted to show. 
\end{proof}

The reader might worry that this argument shows too
much and can be used to contradict the finiteness of $\TP(X)$ over $\TP(k)$
given the spectral sequence of~\cite[Theorem~6.8]{hesselholt-tp}.
However, it is the periodicity of $\TP(k)$ that saves the day. When we
periodicize the spectral sequence, more terms appear so that the differential
hitting $\E^2_{-2,1}$ is the same as that hitting $\E^2_{-4,3}$. This fixes the
non-finiteness.

\bibliographystyle{halpha}
\bibliography{tp}

\end{document}